\documentclass[twoside]{amsart}
\usepackage{amsmath,amsthm,amssymb}

\theoremstyle{plain}
\newtheorem{theorem}{Theorem}[section]
\newtheorem{proposition}[theorem]{Proposition}
\newtheorem{lemma}[theorem]{Lemma}
\newtheorem{corollary}[theorem]{Corollary}
\newtheorem{conjecture}[theorem]{Conjecture}
\theoremstyle{definition}

\newtheorem{remark}[theorem]{Remark}
\theoremstyle{remark}

\newtheorem{acknowledgments}{Acknowledgments}

\renewcommand{\d}{{\bf d}}

\renewcommand{\det}{{\rm det}}

\begin{document}
\title[Hodge-Stickelberger polygons for Laurent polynomials $P(x^s)$]{%
Hodge-Stickelberger polygons for $L$-functions of exponential sums of $P(x^s)
$}
\author{R\'egis Blache, \'Eric F\'erard and Hui June Zhu}
\address{Laboratoire AOC, IUFM de la Guadeloupe,
97139 Les Abymes
}
\email{rblache@iufm.univ-ag.fr}
\address{ Laboratoire GAATI, Universit\'e de la Polyn\'esie Fran\c{c}aise,
Tahiti}
\email{ferard@upf.pf}
\address{Department of mathematics, State University of New York, Buffalo,
NY 14260-2900, USA}
\email{zhu@cal.berkeley.edu}
\date{September 18th, 2007}
\subjclass[2000]{11,14}

\begin{abstract}
Let ${\mathbb{F}}_q$ be a finite field of cardinality $q$ and characteristic 
$p$. Let $\overline{P}(x)$ be any one-variable Laurent polynomial over ${%
\mathbb{F}}_q$ of degree $(d_1,d_2)$ respectively and $p\nmid d_1d_2$. For
any fixed $s\geq 1$ coprime to $p$, we prove that the $q$-adic Newton
polygon of the $L$-functions of exponential sums of $\overline{P}(x^s)$ has
a tight lower bound which we call Hodge-Stickelberger polygon, depending
only on the $d_1,d_2,s$ and the residue class of $(p\bmod s)$. This
Hodge-Stickelberger polygon is a certain weighted convolution of the Hodge
polygon for $L$-function of exponential sums of $\overline{P}(x)$ and the
Newton polygon for the $L$-function of exponential sums of $x^s$ (which is
precisely given by the classical Stickelberger theory). We have an analogous
Hodge-Stickelberger lower bound for multivariable Laurent polynomials as
well.

For any $\nu\in({\mathbb{Z}}/s{\mathbb{Z}})^\times$, we show that there
exists a Zariski dense open subset $ {\mathcal{U}} _\nu$ defined over ${
\mathbb{Q}}$ such that for every Laurent polynomial $P$ in $ {\mathcal{U}}
_\nu(\overline{\mathbb{Q}})$ the $q$-adic Newton polygon of $L(\overline{P}
(x^s)/{\mathbb{F}}_q;T)$ converges to the Hodge-Stickelberger polygon as $p$
approaches infinity and $p\equiv \nu\bmod s$.

As a corollary, we obtain a tight lower bound for the $q$-adic Newton
polygon of the numerator of the zeta function of an Artin-Schreier curve
given by affine equation $y^p-y=\overline{P}(x^s)$. This estimates the $q$%
-adic valuations of reciprocal roots of the zeta function of the
Artin-Schreier curve.
\end{abstract}

\maketitle

\section{Introduction}

\label{S:1}

Let ${\mathbb{A}}_{d_1,d_2}$ be the space of all Laurent polynomials in one
variable $x$ of degree $(d_1,d_2)$ (in $x$ and $x^{-1}$ respectively) where $%
d_1,d_2\geq 1$. It is a rational function with two poles at $\infty$ and $0$%
. The one-pole polynomial case (i.e., $d_2=0$) will also be considered along
the line. For our purpose, we may assume that each Laurent polynomial is
monic at $x^{d_1}$ and hence the coefficient space ${\mathbb{A}}_{d_1,d_2}$
is an irreducible affine space of dimension $d_1+d_2$. In this paper $p$ is
a prime coprime to $d_1d_2$. Let $E(x)$ be the Artin-Hasse exponential
function, namely, $E(x)=\exp(\sum_{i=0}^{\infty}x^{p^i}/p^i)$. Let $\gamma$
be a $p$-adic root of $\log(E(x))$ in the algebraic closure of ${\mathbb{Q}}%
_p$ with $ \mathrm{ord} _p\gamma = 1/(p-1)$. Then $E(\gamma)$ is a primitive 
$p$-th root of unity, which we fix for the rest of the paper and denote it
by $\zeta_p$.

Let $a$ be a positive integer and $q=p^a$. Let $\overline{P}(x)$ be a
rational function on the projective line with two poles of order $d_1$ and $%
d_2$ respectively. Up to an isomorphism over $\overline{\mathbb{F}}_p$ we
may assume the poles are at $\infty$ and $0$ and write 
\begin{equation*}
\overline{P}(x)=\sum_{i=-d_2}^{d_1}\overline{a}_i x^i 
\end{equation*}
where $\overline{a}_i$ lies in ${\mathbb{F}}_q$ and $\overline{P}\in {%
\mathbb{A}}_{d_1,d_2}({\mathbb{F}}_q)$. For any positive integer $k$, let $%
\psi_{q^k}:{\mathbb{F}}_{q^k}\rightarrow {\mathbb{Q}}(\zeta_p)^\times$ be a
nontrivial additive character of ${\mathbb{F}}_{q^k}$ and we fix $%
\psi_{q^k}(\cdot)=\zeta_p^{ \mathrm{Tr} _{{\mathbb{F}}_{q^k}/{\mathbb{F}}%
_p}(\cdot)}$. The $k$-th exponential sum of $\overline{P}(x)\in{\mathbb{F}}%
_q[x,x^{-1}]$ is $S_k(\overline{P})=\sum_{x\in{\mathbb{F}}%
_{q^k}^\times}\psi_{q^k}(\overline{P}(x))$. The $L$-function of the
exponential sum of $\overline{P}$ is defined by 
\begin{equation*}
L(\overline{P}(x); T) = \exp(\sum_{k=1}^{\infty}S_k(\overline{P})\frac{T^k}{k%
}). 
\end{equation*}
It is known that 
\begin{equation*}
L(\overline{P}(x)/{\mathbb{F}}_q;T)=1+b_1T+\cdots+b_{d_1+d_2}T^{d_1+d_2} \in 
{\mathbb{Z}}[\zeta_p][T]. 
\end{equation*}
The most important information about the $L$-function is its reciprocal
roots. They are Weil $q$-numbers, i.e., algebraic integers all Galois
conjugates are of absolute value $\sqrt{q}$. This paper concerns their $q$%
-adic absolute value. This can be effectively studied in terms of $q$-adic
Newton polygon of the $L$-function. The $q$-adic Newton polygon $ \mathrm{NP}
_q(\overline{P}(x);{\mathbb{F}}_q)$ of this $L$-function is defined as the
lower convex hull of the points $(i, \mathrm{ord} _q(b_i))_{i\geq 0}$ on the 
$(x,y)$-plane. Results about this Newton polygon can be found in \cite%
{Wan:1, Zhu:1, Zhu:2}. This polygon is independent of the choice of base
field ${\mathbb{F}}_q$ in $\overline{\mathbb{F}}_p$ (even though the
reciprocal roots of the $L$-function do depend on ${\mathbb{F}}_q$). The
relation between $q$-adic valuation of roots of a polynomial and its $q$%
-adic Newton polygon is explained in details in \cite[Chapter IV]{Koblitz}.

We fix once and for all a positive integer $s\geq 1$. All primes $p$ we
consider will be assumed prime to $s$. The main subject of study of this
paper is $L(\overline{P}(x^s)/{\mathbb{F}}_q;T)$ and its reciprocal roots.
Let $\sigma$ be the permutation on the set $\{0,\dots,s-1\}$ induced by
multiplication of $p$ modulo $s$. We write its cycle decomposition $%
\sigma=\prod_{i=1}^{u} \sigma_i$ for $\ell_i$-cycles $\sigma_i$ (including $1
$-cycles). Let 
\begin{equation*}
\lambda_i:=\frac{\sum_{j\in \sigma_i}j}{ s \ell_i}. 
\end{equation*}
So $0\leq \lambda_i< 1$. Note that $\ell_i$ and $\lambda_i$ are invariants
depending only on $s$, $\nu$ (defined as the least residue of $p$ modulo $s$%
) and the cycle $\sigma_i$, but independent of $p$. See Section \ref{S:3.1}
for more details. Note that for $s|(q-1)$ one recovers the classical formula 
$\lambda_i = \frac{s_p((q-1)r/s)}{a(p-1)},$ where $s_p(n)$ denotes the sum of 
$p$-adic expansions of the integer $n$.

We now define $ \mathrm{HS} ({\mathbb{A}}_{d_1,d_2},\nu,s)$, the \emph{%
Hodge-Stickelberger} polygon of $L(\overline{P}(x^s)/{\mathbb{F}}_q;T)$, as
the polygon with line segments of slopes and lengths 
\begin{eqnarray}  \label{E:HS}
(\frac{m+1-\lambda_i}{d_1}, \ell_i)_{1\leq i\leq u,0\leq m\leq d_1-1}; \quad
(\frac{m+\lambda_i}{d_2},\ell_i)_{1\leq i\leq u, 0\leq m\leq d_2-1}.
\end{eqnarray}
Note that this polygon contains segments $(0,1)$ and $(1,1)$, and it is
symmetric in the sense that for every slope $\alpha$ there is a slope $%
1-\alpha$ of equal length (note that if $\sigma_i$ is the cycle containing $%
r>0$, $\sigma_j$ the one containing $s-r$, then $\lambda_i+\lambda_j=1$).
This polygon depends only on $d_1,d_2,\nu,s$ and is of total horizontal
length $s(d_1+d_2)$.

If $d_2=0$ then the Hodge-Stickelberger polygon is given by the first half
of the line segments in (\ref{E:HS}) minus the segment $(1,1)$.

\begin{remark}
\label{R:2} Consider the Gauss sum over ${\mathbb{F}}_q$ defined by 
\begin{equation*}
G_{{\mathbb{F}}_q}(\psi_q,\chi_s^r):= - \sum_{x\in {\mathbb{F}}_q^\times}
\psi_q(x)\chi_s^{-r}(x)
\end{equation*}
(where $\chi_s$ is a multiplicative character of order $s$ on ${\mathbb{F}}%
_{q}^\times$). The Stickelberger's theorem ( see \cite[Theorem 11.2.1]{B-E-W}
or \cite{Washington:cyclo}) says that $ \mathrm{ord} _q(G_{{\mathbb{F}}%
_q}(\psi_q,\chi_{s}^i))= \lambda_{i}$. In fact, one can show that $L(x^s;{%
\mathbb{F}}_p) = \prod_i(1-T^{\ell_i} G_{{\mathbb{F}}_{p^{\ell_i}}}(\psi_{p^{\ell_i}},%
\chi_{s}^i))$ where $i$ ranges over all distinct cycles in $\sigma$ (see 
\cite{Katz:79} or \cite{Wan:2}). Thus the exact shape of the $p$-adic Newton
polygon of $L(x^s;{\mathbb{F}}_p)$ consists of line segments $%
(\lambda_i,\ell_i)_{2\leq i\leq u}$ (by omitting the 1-cycle $\sigma_1=(0)$).
\end{remark}

By the remark above, our Hodge-Stickelberger polygon can be considered as a
weighted convolution of the Hodge polygon $ \mathrm{HP} ({\mathbb{A}}%
_{d_1,d_2})$ of the $L$-function $L(\overline{P}/{\mathbb{F}}_q;T)$ 
(see \cite{Li-Zhu:1} for details)
and the Newton polygon of $L(x^s/{\mathbb{F}}_q;T)$. The following theorem states
that it gives a lower bound of the $q$-adic Newton polygon of $L$-function.
We use $\succ$ to denote one polygon lies over the next one and their
endpoints meet.

\begin{theorem}
\label{T:1} For any Laurent polynomial $\overline{P} \in {\mathbb{A}}%
_{d_1,d_2}({\mathbb{F}}_q)$, we have 
\begin{equation*}
\mathrm{NP} _q(\overline{P}(x^s);{\mathbb{F}}_q) \succ  \mathrm{HS} ({%
\mathbb{A}}_{d_1,d_2},\nu,s).
\end{equation*}
These two polygons coincide if and only if $p\equiv 1\bmod  \mathrm{lcm }
(sd_1,sd_2)$.
\end{theorem}

If $\overline{P}$ has only one pole of order $d_1\geq 1$ (and $d_2=0$), then
the above two polygons coincide if and only if $p\equiv 1\bmod sd_1$ or $%
d_1=1$.

In fact, we have an analogous result for multivariable Laurent polynomial
which is stated in Section \ref{S:6}.

\begin{remark}
\label{R:1} {}From \cite{Zhu:3} we know that $ \mathrm{NP} _q(\overline{P}%
(x^s);{\mathbb{F}}_q)\succ  \mathrm{HP} ({\mathbb{A}}_{sd_1,sd_2})$, the
latter is the concatenation of the following slopes 
\begin{equation*}
0,1,\frac{1}{sd_1},\ldots,\frac{sd_1-1}{sd_1}, \frac{1}{sd_2},\ldots, \frac{%
sd_2-1}{sd_2}
\end{equation*}
in nondecreasing order each of horizontal length $1$. Hence it is of total
horizontal length $s(d_1+d_2)$. We easily see the following relation 
\begin{equation*}
\mathrm{NP} _q(\overline{P}(x^s);{\mathbb{F}}_q)\succ  \mathrm{HS} ({\mathbb{%
A}}_{d_1,d_2},\nu,s)\succ  \mathrm{HP} ({\mathbb{A}}_{sd_1,sd_2}). 
\end{equation*}
Furthermore, $ \mathrm{HS} ({\mathbb{A}}_{d_1,d_2},\nu,s) =  \mathrm{HP} ({%
\mathbb{A}}_{sd_1,sd_2})$ if and only if $\nu=1$, that is $p\equiv 1\bmod s$.
\end{remark}

Then we examine, as $p$ varies, the asymptotic behavior of the polygons $ 
\mathrm{NP} (P(x^s)\bmod  {\mathcal{P}} )$ where $ {\mathcal{P}} $ is a
prime over $p$. Note that this polygon is independent of the choice of $ {%
\mathcal{P}} $ and so for ease of notation we may consider ${\mathbb{F}}_q$
the residue field of $ {\mathcal{P}} $. It is known (see \cite{Li-Zhu:1}, 
\cite{Zhu:2}) that when $p$ approaches infinity, there is a Zariski dense
open subset $ {\mathcal{U}} $ defined over ${\mathbb{Q}}$ of the space of
rational functions with prescribed poles and polar degrees such that for any
rational function lying in $ {\mathcal{U}} (\overline{\mathbb{Q}})$, the
Newton polygon $ \mathrm{NP} (P(x)\bmod  {\mathcal{P}} )$ tends to the
associated Hodge polygon $ \mathrm{HP} ({\mathbb{A}}_{d_1,d_2})$. For $s>2$
such limit does not exist since there is one distinct Hodge-Stickleberger
polygon for each residue class of prime $p$ in $({\mathbb{Z}}/s{\mathbb{Z}}%
)^\times$ and for $p\equiv 1\bmod  \mathrm{lcm } (sd_1,sd_2)$ the Newton
polygon coincides with the Hodge-Stickelberger polygon. See more discussion
on this topic in Section \ref{S:6}. In our main result we show that in each
fixed residue class of primes, the situation is similar to the case $s=1$.

\begin{theorem}
\label{T:2} For every integer $1\leq \nu \leq s-1$ coprime to $s$, there
exists a Zariki dense open subset $ {\mathcal{U}} _\nu$ in ${\mathbb{A}}%
_{d_1,d_2}$ defined over ${\mathbb{Q}}$ where $d_1\geq 1$ and $d_2\geq 0$,
such that for any $P(x)$ lying in $ {\mathcal{U}} _\nu(\overline{{\mathbb{Q}}%
})$, we have 
\begin{equation*}
\lim_{p\rightarrow \infty,p\equiv \nu\bmod s}  \mathrm{NP} (P(x^s)\bmod  {%
\mathcal{P}} ) =  \mathrm{HS} ({\mathbb{A}}_{d_1,d_2},\nu,s)
\end{equation*}
for all primes $ {\mathcal{P}} $ over $p$.
\end{theorem}

These theorems about exponential sums have applications to the Zeta function
of Artin-Schreier curves over ${\mathbb{F}}_q$, namely the projective curves 
$C$ defined by affine equation $y^p-y=\overline{P}(x)$ over ${\mathbb{F}}_q$%
. It is well known that all reciprocal roots of the numerator of the Zeta
function of $C$ are eigenvalues of Frobenius endomorphism, and they are Weil 
$q$-numbers. The following corollary estimates the $q$-adic absolute values
of these reciprocal roots. We explore it via the $q$-adic Newton polygon $ 
\mathrm{NP} _q(C/{\mathbb{F}}_q)$, defined as the $q$-adic Newton polygon of
the numerator of the Zeta function of $C$. In this paper a constant $c$
multiple of a polygon means that we amplify or shrink each slope length by a
factor of $c$ horizontally and vertically.

\begin{corollary}
\label{C:1} (i) Let $ \mathrm{NP} (C_s/{\mathbb{F}}_q)$ be the $q$-adic
Newton polygon of the Artin-Schreier curve $C_s: y^p-y=\overline{P}(x^s)$
over ${\mathbb{F}}_q$. Then $\frac{1}{p-1} \mathrm{NP} (C_s/{\mathbb{F}}%
_q)\succ  \mathrm{HS} ({\mathbb{A}}_{d_1,d_2},\nu,s)$. These two polygons
coincide if and only if $p\equiv 1\bmod  \mathrm{lcm } (sd_1,sd_2)$. If $%
\overline{P}$ has only one pole of degree $d_1\geq 1$ (and $d_2=0$) then the
two polygons coincide if and only if $p\equiv 1\bmod sd_1$ or $d_1=1$.

(ii) For every integer $1\leq \nu\leq s-1$ coprime to $s$, there exists a
Zariski dense open subset $ {\mathcal{U}} _\nu$ in ${\mathbb{A}}_{d_1,d_2}$,
defined over ${\mathbb{Q}}$, such that for any $P(x)$ lying in $ {\mathcal{U}%
} _\nu(\overline{{\mathbb{Q}}})$, we have 
\begin{equation*}
\lim_{p\rightarrow \infty,p\equiv \nu\bmod s}\frac{1}{p-1}  \mathrm{NP} (C_s%
\bmod  {\mathcal{P}} ) =  \mathrm{HS} ({\mathbb{A}}_{d_1,d_2},\nu,s) 
\end{equation*}
for any prime $ {\mathcal{P}} $ over $p$.
\end{corollary}

\begin{proof}
Results in Theorems \ref{T:1} and \ref{T:2} can be translated directly to
this corollary by using the same argument as that in \cite[Corollary 1.3]%
{Zhu:3}.
\end{proof}

\begin{remark}
\label{R:3} We remark that in the above corollary, one may replace the
affine coefficient space ${\mathbb{A}}_{d_1,d_2}$ by the moduli space ${%
\mathcal{AS}}_g$ of Artin-Schreier curves of genus $g:=\frac{(p-1)(d_1+d_2)}{%
2}$ as defined in \cite{PZ:1}.
\end{remark}

We conclude this section by providing general notation and outline for the
rest of this paper. Throughout the entire paper, we fix integers $%
d_1,d_2,s\geq 1$. We consider prime numbers $p$ that are always coprime to $s
$. We always assume the residue field of the prime ideal $ {\mathcal{P}} $
is ${\mathbb{F}}_q$, where $q$ is a $p$-power and we write $q=p^a$. The
permutation $\sigma$ is induced on the set $\{0,\dots,s-1\}$ by
multiplication of $p$ modulo $s$. We always write its cycle decomposition as 
$\sigma=\prod_{i=1}^{u}\sigma_i$ including $1$-cycles. Finally, we denote by 
$E(x)$ be the $p$-adic Artin-Hasse exponential function. Our main theorems %
\ref{T:1} and \ref{T:2} are proved at the end of Section \ref{S:5}. Similar
result on twisted exponential sums is given in Propositions \ref{P:twisted}
and \ref{P:1} of Sections \ref{S:3} and \ref{S:4} respectively. At the end
of the paper in section 6 we discuss some open questions and give statement
of multivariable cases analog of Theorem \ref{T:1}.

\section{Two lemmas about nuclear matrices}

\label{S:2}

To make the proofs of our results as smooth as possible, we summarize some
fringe results here. These results will be employed in Sections \ref{S:3}
and \ref{S:4}. The reader may wish to skip this section at first reading.

Let $K$ be any complete non-Archimedean field with $p$-adic valuation $%
|\cdot|_p$. We refer the readers to \cite{Serre} for basic facts about
Serre's theory of completely continuous maps and Fredholm determinants. For
any $K$-Banach spaces $V$ and $V^{\prime}$ that admit orthonormal basis,
denote by $ {\mathcal{C}} (V,V^{\prime})$ the set of completely continuous $K
$-linear maps from $V$ to $V^{\prime}$. We say that a matrix $M$ over $K$ is 
\emph{nuclear} if there exists a $K$ Banach space $V$ and a $u$ in $ {%
\mathcal{C}} (V,V)$ such that $M$ is the matrix of $u$ with respect to some
orthonormal basis of $V$. If $M=(m_{ij})_{i,j\geq 1}$ is a matrix over $K$,
then $M$ is nuclear if and only if $\lim_{i\rightarrow\infty}{\inf_{j\geq 1} 
\mathrm{ord} _p m_{ij}}=+\infty$.

\begin{lemma}
\label{L:transfer} Let $\vec{M}=(M_0,\, M_1, \cdots, M_{a-1})$ be an $a-$%
tuple of nuclear matrices over ${\mathbb{C}}_p$. Set the block matrix 
\begin{eqnarray*}
\vec{M}_{[a]}:= 
\begin{pmatrix}
0 &  & \cdots & 0 & M_{a-1} \\ 
M_0 & 0 &  &  & 0 \\ 
0 & M_1 & 0 &  & \vdots \\ 
\vdots &  & \ddots & 0 &  \\ 
0 & \cdots & 0 & M_{a-2} & 0%
\end{pmatrix}%
.
\end{eqnarray*}
Then $\det(1-(M_{a-1}\cdots M_1M_0)T^a)=\det(1-\vec{M}_{[a]}T)$.
\end{lemma}

\begin{proof}
See \cite[Section 5]{Li-Zhu:1}.
\end{proof}

\begin{lemma}
\label{L:main} Let $\{M_t\}_{t=0,\ldots,a-1}$ be any nuclear matrices over $K
$. Let $ {\mathcal{A}} _t$ be the set of all $k\times k$ submatrices in $M_t$%
. Fix an integer $k\geq 1$ and let $c_k$ be the coefficient of $T^k$ in $%
\det(1-M_{a-1}M_{a-2}\cdots M_0 T)$. Then we have $ \mathrm{ord} _p c_k \geq
\sum_{t=0}^{a-1} \inf_{W_t\in  {\mathcal{A}} _t} \mathrm{ord} _p (\det W_t)$.
\end{lemma}

\begin{proof}
By Lemma \ref{L:transfer}, $c_k$ is the coefficient of $T^{ak}$ in the $T$%
-adic expansion of $\det(1-\vec{M}_{[a]}T)$, which is the infinite sum of $%
(-1)^{ak}\det N$ where $N$ runs over all principal $ak\times ak$ submatrices
in $\vec{M}_{[a]}$. Let $N$ be such a matrix, and let $N_t$ be the
intersection of $N$ and $M_t$ as submatrices of $\vec{M}_{[a]}$ for all $%
0\leq t\leq a-1$. It is easy to see that $\det N
=(-1)^{(ak-1)k}\prod_{t=0}^{a-1} \det N_t $ or $0$ depending on whether
every $N_t$ is a $k\times k$ submatrix of $M_t$ or not. So for $p$-adic
evaluation purpose, we may assume every $N_t$ is a $k\times k$ matrix. Think
of $N_t$ as a submatrix of $M_t$ from now on and $N_t\in  {\mathcal{A}} _t$.
Our assertion follows immediately.
\end{proof}

\section{$L$-functions of twisted exponential sums}

\label{S:3}

In this section we assume $s|(q-1)$. Let $k\geq 1$. Let $\chi_s$ be a
multiplicative character of order $s$ defined on ${\mathbb{F}}_{q^k}^\times$%
. We fix it as $\chi_s = \chi\circ N_{{\mathbb{F}}_{q^k}/{\mathbb{F}}%
_q}(\cdot)$ where $\chi$ is a multiplicative character of order $s$ on ${%
\mathbb{F}}_q^\times$.

Fix an integer $0\leq r \leq s-1$, let $\sigma_i$ be the cycle of $\sigma$
containing $r$, and $\lambda:=\lambda_i=\sum_{j\in\sigma_i}j/(s \ell_i)$.
For any Laurent polynomial $\overline{P}(x)$ in ${\mathbb{A}}_{d_1,d_2}({%
\mathbb{F}}_q)$, define the $L$-function

\begin{eqnarray}  \label{E:L-twist1}
L(\overline{P}(x)/{\mathbb{F}}_q,\chi_s^r;T) :=\exp(\sum_{k=1}^{\infty}S_k(%
\overline{P},\chi_s^r)\frac{T^k}{k}).
\end{eqnarray}
where $S_k(\overline{P},\chi_s^r)= \sum_{x\in {\mathbb{F}}%
_{q^k}^\times}\psi_{q^k}(\overline{P}(x))\chi_s^r(x)$.

{}From Weil's theorem, this $L$-function is a polynomial of degree $d_1+d_2$
and its reciprocal roots in ${\mathbb{C}}$ are algebraic integers with
Archimedean absolute value $q^{1/2}$ and $\ell$-adic absolute value $1$ for
any prime $\ell\neq p$. We shall study the $q$-adic absolute value of these
reciprocal roots. We denote by $ \mathrm{NP} _q(\overline{P},\chi_s^r;{%
\mathbb{F}}_q)$ the Newton polygon of $L(\overline{P}/{\mathbb{F}}%
_q,\chi_s^r;T)$ defined analogously as that for $ \mathrm{NP} _q(\overline{P}%
;{\mathbb{F}}_q)$.

\subsection{Twisted Hodge-Stickelberger polygons}

\label{S:3.1}

Denote by $ \mathrm{HS} ({\mathbb{A}}_{d_1,d_2},\nu,\chi_s^r)$ the \emph{%
twisted Hodge-Stickelberger polygon} of multiplicative character $\chi_s^r$
with slopes and lengths 
\begin{equation*}
\{(\frac{m+1-\lambda}{d_1},1)_{0\leq m\leq d_1-1};\quad (\frac{m+\lambda}{d_2%
},1)_{0\leq m\leq d_2-1}\} 
\end{equation*}
It is of total horizontal length $d_1+d_2$. This polygon can be found in the
literature, for example, see \cite[Theorem 3.20]{AS:1} and \cite[Corollary
3.18]{AS:2}. In the polynomial case,i.e., $d_2=0$, the twisted
Hodge-Stickelberger polygon consists of the first half of the above line
segments minus the segment $(0,1)$ and is of horizontal length $d_1-1$.

\begin{remark}
\label{R:4} The twisted Hodge-Stickelberger polygon $ \mathrm{HS} ({\mathbb{A%
}}_{d_1,d_2},\nu,\chi_s^r)$ we give above coincides with the Hodge polygon
defined in \cite[Corollary 3.18]{AS:2} in one-variable case. We shall verify
this explicitly below. Set $\d:=-(q-1)r/s$ in notation of \cite{AS:2}. Then $%
\d^{(v)} =-(q-1)\sigma^{v}(r)/s$, and for any $-d_2+1\leq j\leq d_1$, we
have 
\begin{equation*}
u_{\d^{(v)}}(j)=x^{\frac{\d^{(v)}}{q-1}+j},~\mbox{\rm and}~w(u_{\d^{(v)}%
}(j)) = \left\{ 
\begin{array}{ll}
\frac{\frac{\d^{(v)}}{q-1}+j}{d_1} =\frac{j}{d_1}-\frac{\sigma^{v}(r)}{sd_1}
& \mbox{if } j>0, \\ 
\frac{\frac{\d^{(v)}}{q-1}+j}{d_2} =-\frac{j}{d_2}+\frac{\sigma^{v}(r)}{sd_2}
& \mbox{if } j\leq 0.%
\end{array}
\right.
\end{equation*}
These are due to the fact that the weight of $x^r$ is $r/d_1$ when $r\geq 0$
and $-r/d_2$ when $r\leq 0$ in our case. The Hodge polygon slopes $b_j$
defined in \cite[above Theorem 3.17]{AS:2} can be expressed as 
\begin{equation*}
b_j= \left\{ 
\begin{array}{ll}
\frac{1}{a} \sum_{v=0}^{a-1}\left(\frac{j}{d_1}-\frac{\sigma^{v-a}(r)}{sd_1}%
\right) =\frac{j-\lambda}{d_1} & \mbox{if } j>0, \\ 
\frac{1}{a}\sum_{v=0}^{a-1}\left(\frac{-j}{d_2} +\frac{\sigma^{v-a}(r)}{sd_2}%
\right) =\frac{-j+\lambda}{d_2} & \mbox{if } j\leq 0%
\end{array}
\right. 
\end{equation*}
These yield exactly the slopes of our twisted Hodge-Stickelberger polygon $ 
\mathrm{HS} ({\mathbb{A}}_{d_1,d_2},\nu,\chi_s^r)$ defined above.
\end{remark}

We use $\boxplus$ to denote the concatenation of line segments which are
given via pairs of slopes and horizontal length so that the slopes are in
non-decreasing order. Now we have the splitting of the Hodge-Stickelberger
polygon into twisted Hodge-Stickelberger polygons below in the lemma.

\begin{lemma}
\label{L:HS-split} (i) We have 
\begin{equation*}
\mathrm{HS} ({\mathbb{A}}_{d_1,d_2},\nu,s) = \boxplus_{i}\ell_i\;  \mathrm{HS%
} ({\mathbb{A}}_{d_1,d_2},\nu,\chi_s^{r_i}) 
\end{equation*}
where the box-sum ranges in the distinct cycles $\sigma_i$ of $\sigma$, and
for each $i$ $r_i$ is a representative in $\sigma_i$.

(ii) If $\nu=1$ then $ \mathrm{HS} ({\mathbb{A}}_{d_1,d_2},1,s) =  \mathrm{HP%
} ({\mathbb{A}}_{sd_1,sd_2})$.
\end{lemma}

\begin{proof}
The first statement is clear by the definition of $ \mathrm{HS} ({\mathbb{A}}%
_{d_1,d_2},\nu,s)$ in (\ref{E:HS}). For the second statement, one only needs
to recognize that for $p\equiv 1\bmod s$ we have $\ell_i=\ell_i^{\prime}=1$
for every $i$ and $\lambda_i=r/s$ for every $0\leq r \leq s-1$ in $\sigma_i$%
. The rest is explicit and elementary calculation.
\end{proof}

\subsection{Trace formula for twisted exponential sums}

\label{S:3.2}

Let ${\mathbb{Q}}_q$ denote the unique unramified extension of ${\mathbb{Q}}%
_p$ of degree $a$ and ${\mathbb{Z}}_q$ its ring of integers. Let $\Omega_1:={%
\mathbb{Q}}_p(\zeta_p)$ and let $\Omega_a$ the unique unramified extension
of $\Omega_1$ of degree $a$ in ${\mathbb{C}}_p$. Recall that $\gamma\in
\Omega_1$ such that ${\mathbb{Z}}_p[\gamma]={\mathbb{Z}}_p[\zeta_p]$. Fix
roots $\gamma^{1/d_1}$ and $\gamma^{1/d_2}$ in ${\mathbb{C}}_p$, we denote
by $\Omega_1^{\prime}=\Omega_1(\gamma^{1/d_1},\gamma^{1/d_2})$ and $%
\Omega_a^{\prime}=\Omega_1^{\prime}\Omega_a$. Below we denote by $K$ ($%
K^{\prime}$ respectively) a complete non-Archimedean field containing $%
\Omega_a$ ($\Omega_a^{\prime}$ respectively).

By taking Teichm\"uller lifts of coefficients of $\overline{P}\in{\mathbb{F}}%
_q[x,x^{-1}]$, we get $\hat{P}(x)=\sum_{i=-d_2}^{d_1} \hat{a_i} x^i\in{%
\mathbb{Z}}_q[x,x^{-1}]$. Note that $\hat{a_i}^q=\hat{a_i}$ and $\hat{a_i}%
\equiv \overline{a}_i\bmod  {\mathcal{P}} $ where $ {\mathcal{P}} $ is the
prime ideal in $\Omega_a$ lying over $p$. For any $0<\rho<1$ in $|K|_p$ let $
{\mathcal{H}} _\rho(K)$ be the ring of rigid analytic functions over $K$ on
the annulus with $\rho\leq |x|_p\leq 1/\rho$. It is a $p$-adic Banach space
with the natural $p$-adic supremum norm.

Let the operator $U_q$ on $ {\mathcal{H}} _\rho$ be defined by $(U_q\xi)(X):=%
\frac{1}{q} \sum_{Z^q=X}\xi(Z)$ for any $\xi\in {\mathcal{H}} _\rho$. If $%
\xi(X)=\sum_{i=-\infty}^{\infty}c_iX^i$ then $U_q(\xi)=\sum_{i=-\infty}^{%
\infty}c_{iq}X^i$. Let $\tau$ be a lifting of the Frobenius of $\overline{{%
\mathbb{F}}}_p$ to $K$ such that $\tau(\gamma)=\gamma$. Define three
elements in $ {\mathcal{H}} _\rho(K)$ below 
\begin{eqnarray}
F(X)&=&\prod_{i=-d_2}^{d_1}E(\gamma\hat{a_i} X^i), \\
F_{[a]}(X)&=&\prod_{t=0}^{a-1}F^{\tau^t}(X^{p^t}), \\
H(X)&=& X^{\frac{(q-1)r}{s}}F_{[a]}(X).
\end{eqnarray}
These above are all power series in ${\mathbb{Z}}_p[\gamma][\vec{\hat{a_i}}%
][[X]]$ and hence in ${\mathbb{Z}}_q[\zeta_p][[X]]$. Let $\alpha:=U_q\circ
H(X)$ by which we mean the composition map of $U_q$ with the multiplication
map by $H(X)$. Then $\alpha$ is a completely continuous $K$-linear
endomorphism of $ {\mathcal{H}} _\rho(K)$ for some suitable $0<\rho<1$.

\begin{lemma}
We have 
\begin{eqnarray}  \label{E:trace-formula}
L(\overline{P}/{\mathbb{F}}_q,\chi_s^r;T) &=& \frac{\det(1-T\alpha)}{%
\det(1-Tq\alpha)}
\end{eqnarray}
and it is a polynomial in ${\mathbb{Z}}[\zeta_p,\zeta_s][T]$ of degree $%
d_1+d_2$.
\end{lemma}

\begin{proof}
The rationality is a routine consequence of the Dwork-Monsky-Reich trace
formula so we omit its proof here. The assertion of its degree follows from 
\cite{Denef-Loeser} (or \cite{AS:2}).
\end{proof}

\subsection{$p$-adic estimate of twisted exponential sums}

\label{S:3.3}

Let $0\leq r\leq s-1$. Write the $p$-adic expansion 
\begin{eqnarray}  \label{E:Kt-2}
(q-1)r/s &=& \sum_{t=0}^{a-1}K_t p^t
\end{eqnarray}
for $0\leq K_t\leq p-1$. Then we have 
\begin{eqnarray}  \label{E:Kt}
\lambda &= \frac{\sum_{t=0}^{a-1}K_t}{a(p-1)}& = \frac{s_p((q-1)r/s)}{a(p-1)}
\end{eqnarray}
where $s_p(\cdot)$ denotes the sum of $p$-adic expansions. 

Let $F_t(X)=X^{K_t}F^{\tau^t}(X)$ and 
\begin{eqnarray*}
\alpha_t &:=& U_p \circ F_t(X).
\end{eqnarray*}

\begin{lemma}
\label{L:factorization} The maps $\alpha_t$ are completely continuous $K$%
-linear endomorphisms of $ {\mathcal{H}} _\rho(K)$ for some suitable $%
0<\rho<1$. We have 
\begin{eqnarray}  \label{E:1}
\alpha = \alpha_{a-1}\circ \cdots \circ\alpha_1\circ\alpha_0.
\end{eqnarray}
\end{lemma}

\begin{proof}
The first statement is Dwork theory. Using $f(x)\circ U_p = U_p\circ f(x^p)$%
, we have by (\ref{E:Kt-2}) 
\begin{eqnarray*}
\alpha_{a-1}\circ\cdots\circ \alpha_0&=& (U_p\circ \cdots \circ U_p)\circ
(X^{\sum_{t=0}^{a-1}K_tp^t} F_{[a]}(X)) \\
&=& U_q\circ H(X) = \alpha.
\end{eqnarray*}
This finishes the proof.
\end{proof}

For any $i\in {\mathbb{Z}}$, consider the $p$-adic Mittag-Leffler
decomposition $F(X)X^i= \sum_{m=-\infty}^{\infty}H^{m,i}X^m$. Write $%
\alpha_t(X^i) = \sum_{m=-\infty}^{\infty}B^{m,i}_tX^m,$ we have $%
B^{m,i}_t=\tau^t H^{mp-K_t,i}$. We know $H^{m,i}, B^{m,i}_t$ lie in ${%
\mathbb{Z}}_p[\gamma][\vec{\hat{a_i}}]$. Then from the $p$-adic valuation of
the coefficients of $E$ (see \cite{Dwork1}) we have 
\begin{equation*}
\mathrm{ord} _p B^{m,i}_t \geq \frac{1}{p-1}\max\left(\frac{pm-K_t-i}{d_1},-%
\frac{pm-K_t-i}{d_2}\right)
\end{equation*}

By $p$-adic Mittag-Leffler decomposition, every element in the $K$-linear
space $ {\mathcal{H}} _\rho(K)$ can be uniquely represented as $%
\sum_{i=-\infty}^{\infty} c_iX^i$ for $c_i\in K$, and so $ {\mathcal{H}}
_\rho(K)$ has a natural monomial basis $\vec{b}_{\mathrm{unw}%
}=\{1,X,X^2,\ldots;X^{-1},X^{-2},\ldots\}$. Let $Z_1=\gamma^{1/d_1}X$ and $%
Z_2=\gamma^{1/d_2}X^{-1}$, then $\vec{b}= \{Z_1,Z_1^2,\ldots;
1,Z_2,Z_2^2,\ldots \} $ forms a basis for $ {\mathcal{H}} _\rho(K^{\prime})$%
. 
Let $M_t$ be the matrix of $\alpha_t$ with respect to the basis $\vec{b}$.
Its entries lie in ${\mathbb{Z}}_q[\gamma^{1/d_1},\gamma^{1/d_2}]$. From now
on, we shall consider the coefficients liftings $\hat{a_i}$ of $P(x)$ as
variables throughout this section, and set $\vec{\hat{a}}=(\hat{a}_i)$, then
the entries of $M_t$ lie in ${\mathbb{Z}}_p[\gamma^{1/d_1},\gamma^{1/d_2}][%
\vec{\hat{a}}]$. Note that $ \mathrm{ord} _p(\cdot)$ and $ \mathrm{ord}
_q(\cdot)$ also denote the natural $p$-adic valuations on the multi-variable
polynomial ring ${\mathbb{Z}}_p[\gamma^{1/d_1},\gamma^{1/d_2}][\vec{\hat{a}}]
$ induced from that on ${\mathbb{Z}}_p$.

We are ready to give estimates for the $p$-adic valuations of the
coefficients of $M_t$. Note that we omit the subscript $t$ in the
coefficients since no confusion can occur.

\begin{lemma}
\label{L:bounds} For all $i\geq 0$ we have $\alpha_t Z_J^i =
\sum_{m=1}^{\infty}C_{1,J}^{m,i}Z_1^m+\sum_{m=0}^{\infty}C_{2,J}^{m,i}Z_2^m$
where $C_{\bigstar}^{m,i}$ are the entries of $M_t$. The lower bounds of $ 
\mathrm{ord} _pC_{\bigstar}^{m,i}$ are

\begin{tabular}{|c||ll|}
\hline
$\mathrm{ord} _p(\cdot)\geq$ & $Z_1^i~(i>0)$ & $Z_2^i~(i\geq 0)$ \\ 
\hline\hline
$Z_1^m~(m>0)$ & \fbox{$\frac{m}{d_1}-\frac{K_t}{d_1(p-1)}$} & $\frac{m}{d_1}-%
\frac{K_t}{d_1(p-1)}+\frac{i}{p-1}(\frac{1}{d_1} +\frac{1}{d_2})$ \\ 
$Z_2^m~(m\geq 0)$ & $\frac{m}{d_2}+\frac{K_t}{d_2(p-1)}+\frac{i}{p-1}(\frac{1%
}{d_1} +\frac{1}{d_2})$ & \fbox{$\frac{m}{d_2}+\frac{K_t}{d_2(p-1)}$} \\ 
\hline
\end{tabular}
\end{lemma}

\begin{proof}
See \cite{Zhu:3} page 1542--1543 for details.
\end{proof}

For any $0\leq t\leq a-1$, let $ {\mathcal{L}} _t$ be the set of rational
numbers $ {\mathcal{L}} _t:= \{ \frac{m}{d_1}- \frac{K_t}{d_1(p-1)} | m\geq
1\}\cup\{ \frac{m}{d_2}+ \frac{K_t}{d_2(p-1)} | m\geq 0\}. $ For every $%
k\geq 1$ let $\delta_t^{(k)}$ the sum of $k$ least numbers in $ {\mathcal{L}}
_t$. Split these $k$ numbers in terms of $j=1$ or $2$ we have $k_1+k_2=k$
such that 
\begin{eqnarray*}
\delta_t^{(k)} &=&\sum_{m=1}^{k_1}(\frac{m}{d_1}-\frac{K_t}{d_1(p-1)})
+\sum_{m=0}^{k_2-1}(\frac{m}{d_2}+\frac{K_t}{d_2(p-1)}).
\end{eqnarray*}
By (\ref{E:Kt}) we have (note that $k_1$ and $k_2$ do not depend on $t$) 
\begin{eqnarray}
\frac{1}{a}\sum_{t=0}^{a-1}\delta_t^{(k)} &=& \frac{k_1(k_1+1)}{2d_1} - 
\frac{k_1\lambda}{d_1} +\frac{k_2(k_2-1)}{2d_2} + \frac{k_2\lambda}{d_2} 
\notag \\
&=& \frac{k_1(k_1-1)}{2d_1} + \frac{k_1(1-\lambda)}{d_1} +\frac{k_2(k_2-1)}{%
2d_2} + \frac{k_2\lambda}{d_2}.  \label{E:delta}
\end{eqnarray}

\begin{lemma}
\label{L:bounds2} For any $k\times k$ submatrix of $M_t$, $W_t$, we have 
\begin{equation*}
\mathrm{ord} _p(\det W_t) \geq \delta_t^{(k)}.
\end{equation*}
\end{lemma}

\begin{proof}
Follows from Lemma \ref{L:bounds}.
\end{proof}

\begin{proposition}
\label{P:twisted} Write $\det(1-\alpha T) = 1+\sum_{k=1}^{\infty} C_kT^k$,
then

\begin{enumerate}
\item[(i)] $ \mathrm{ord} _qC_k \geq \frac{1}{a}\sum_{t=0}^{a-1}%
\delta_t^{(k)}$;

\item[(ii)] $ \mathrm{NP} _q(\overline{P}, \chi_s^r;{\mathbb{F}}_q) \succ 
\mathrm{HS} ({\mathbb{A}}_{d_1,d_2},\nu,\chi_s^r)$ for all $\overline{P}\in{%
\mathbb{A}}_{d_1,d_2}({\mathbb{F}}_q)$;

\item[(iii)] These two polygons coincide if and only if $p\equiv 1\bmod  
\mathrm{lcm } (sd_1,sd_2)$ or $(d_1,d_2)=(1,0)$.
\end{enumerate}
\end{proposition}

\begin{proof}
(i) From the decomposition of $\alpha$ in Lemma \ref{L:factorization} we can
apply the results in Lemma \ref{L:main} to $\det(1-\alpha T)$, and we have 
\begin{eqnarray*}
\mathrm{ord} _q(C_k) =\frac{1}{a} \mathrm{ord} _p(C_k) &\geq & \frac{1}{a}%
\sum_{t=0}^{a-1} \inf_{W_t\in {\mathcal{A}} _t} ( \mathrm{ord} _p\det W_t).
\end{eqnarray*}
The result follows from Lemma \ref{L:bounds2}.

(ii) By the trace formula (\ref{E:trace-formula}), we know that $ \mathrm{NP}
_q(\overline{P},\chi_s^r;{\mathbb{F}}_q)$ is identical to the slope $<1$
part of $ \mathrm{NP} _q(1+C_1T+C_2T^2+\cdots)$ (see \cite{Li-Zhu:1}). The
latter can be identified as the condition that $k_1\leq d_1$, $k_2\leq d_2-1$%
. Thus by part (i) the lower bound of $ \mathrm{NP} _q(\overline{P},\chi_s^r;%
{\mathbb{F}}_q)$ is precisely $ \mathrm{HS} ({\mathbb{A}}_{d_1,d_2},\nu,%
\chi_s^r)$ defined in Section \ref{S:3.1} and by (\ref{E:delta}).

(iii) The sufficiency follows from Lemma \ref{L:HS-split} (ii) and Remark %
\ref{R:2} for the case $(d_1,d_2)=(1,0)$. See \cite[Theorem 1.1]{Zhu:3} for
proof of the converse direction.
\end{proof}

\begin{remark}
\label{R:5} The main result in Proposition \ref{P:twisted}(ii) is known to 
\cite[Corollary 3.18]{AS:2} as we noted in Remark \ref{R:4}. We gave a
different proof here in order to be used in the proof of our result in
Proposition \ref{P:1}.
\end{remark}

\section{Asymptotic behavior of $L(\overline{P}/{\mathbb{F}}_q,\protect\chi%
_s^r;T)$}

\label{S:4}

Here again, we assume $s|(q-1)$. Recall that $ \mathrm{NP} _q(\overline{P}%
,\chi_s^r;{\mathbb{F}}_q)$ denotes the $q$-adic Newton polygon of the $L$%
-function $L(\overline{P}/{\mathbb{F}}_q,\chi_s^r;T)$ of twisted exponential
sums. In this section we shall show that for $p$ large enough in a
congruence class mod $s$, this Newton polygon generically converges to the
corresponding twisted Hodge-Stickelberger polygon. (See Proposition \ref{P:1}
for precise statement.) Below we briefly outline our approach, which is very
similar to that in \cite[Sections 4,5]{Li-Zhu:1} and hence we do not
elaborate.

We fix some integer $k$ with $1\leq k\leq d_1+d_2$ in the following, and we
write $k=k_1+k_2$ as in Section \ref{S:3.3}. Let $M_{t,1}^{[k_1]}$ (\textit{%
resp.} $M_{t,2}^{[k_2]}$) denote the $k_1\times k_1$ (\textit{resp.} $%
k_2\times k_2$) submatrix of $M_t$ defined by 
\begin{equation*}
M_{t,1}^{[k_1]}=\left((C_{1,1}^{m,i})_{1\leq m,i\leq k_1}\right)
~(resp.~M_{t,2}^{[k_2]}=\left((C_{2,2}^{m,i})_{0\leq m,i\leq k_2-1}\right)).
\end{equation*}

In this paper we should sometimes consider the coefficients $a_i$ of $P(x)$
as variables and denote them by the vector $\vec{a}$. The \emph{weight} of a
monomial $\prod_{i=-d_2}^{d_1} a_i^{n_i}$ in $K[\vec{a}]$ is equal to $%
\sum_{i=-d_2}^{d_1} |i|n_i$. A relevant example in this section is that the
minimal weight monomials in $H^{m,i}$ (defined under Lemma \ref%
{L:factorization}) are of weight $|m-i|$; and hence the minimal weight
monomials in $B^{m,i}_t$ are of weight $|mp-K_t-i|$. All minimal weight
monomials in (the formal expansion of) $\det M_t(\vec{a})$ lie in $\gamma^{{%
\mathbb{Q}}}{\mathbb{Q}}[\vec{a}]$. As shown in \cite[Proposition 3.8]%
{Li-Zhu:1}, one can find a monomial in $\det M_{t,1}^{[k_1]}$ (\textit{resp.}
$\det M_{t,2}^{[k_2]}$) of minimal weight that does not cancel out with
others terms. Moreover, the $p$-adic order $s_{1,t}$ (\textit{resp.} $s_{2,t}
$) of the coefficient of this monomial is minimal among the $p$-adic orders
of all monomials in $\det M_{t,1}^{[k_1]}$ (\textit{resp.} $\det
M_{t,2}^{[k_2]}$). This monomial corresponds to a permutation $\rho_{1,t}$ (%
\textit{resp.} $\rho_{2,t}$) in the permutation group $S_{k_1}$ (\textit{%
resp.} $S_{k_2}$). For $J=1,2$, let $r_{J,i,j}$ be the least nonnegative
residue of $-(pi-j) \bmod d_J$. Then we have 
\begin{eqnarray*}
s_{1,t} &=&\frac{k_1(k_1+1)}{2d_1}-\frac{k_1K_t}{d_1(p-1)} +\frac{1}{d_1(p-1)%
}\sum_{i=1}^{k_1}r_{1,i,\rho_{1,t}(i)+K_t}; \\
s_{2,t} &=&\frac{k_2(k_2-1)}{2d_2}+\frac{k_2K_t}{d_2(p-1)} +\frac{1}{d_2(p-1)%
}\sum_{i=0}^{k_2-1}r_{2,i,\rho_{2,t}(i)-K_t}.
\end{eqnarray*}
For each fixed $k$ let 
\begin{eqnarray}  \label{E:s_k}
s_k &:=& \frac{1}{a}\sum_{t=0}^{a-1}(s_{1,t}+s_{2,t}) =\frac{k_1(k_1-1)}{2d_1%
} + \frac{k_2(k_2-1)}{2d_2} +\epsilon_{k,p}
\end{eqnarray}
where 
\begin{equation*}
\epsilon_{k,p}: = \frac{1}{a(p-1)d_1}\sum_{t=0}^{a-1}\sum_{i=1}^{k_1}r_{1,i,%
\rho_{1,t}(i)+K_t} +\frac{1}{a(p-1)d_2}\sum_{t=0}^{a-1}%
\sum_{i=1}^{k_2}r_{2,i,\rho_{2,t}(i)-K_t}. 
\end{equation*}

Let $M_t^{[k]}$ be the $k\times k$ submatrix of $M_t$ defined by the block
matrix 
\begin{equation*}
M_t^{[k]} = \left(%
\begin{array}{l|l}
(C_{1,1}^{m,i})_{1\leq m,i \leq k_1} & (C_{1,2}^{m,i})_{1\leq m \leq
k_1,~0\leq i \leq k_2-1} \\ 
& (C_{2,1}^{m,i})_{0\leq m \leq k_2-1,~1\leq i \leq k_1} \\ \hline
&  \\ 
(C_{2,1}^{m,i})_{0\leq m \leq k_2-1,~1\leq i \leq k_1} & (C_{2,2}^{m,i})_{0%
\leq m,i \leq k_2-1} \\ 
& 
\end{array}%
\right). 
\end{equation*}
Then the terms of minimal valuation in the expansion of $\det M_t^{[k]}$
come from the product $\det M_{t,1}^{[k_1]}\cdot \det M_{t,2}^{[k_2]}$; they
have $p$-adic order equal to $s_{1,t}+s_{2,t}$. The product of the terms of
minimal $p$-adic order in each of the $\det M_t^{[k]}$ gives precisely the
lowest $\gamma$-power term in $\prod_{t=0}^{a-1} \det M_t^{[k]}$. This term
can be written as a product $\gamma^{a(p-1)s_k} U G_{\nu,r}$ for some $p$%
-adic unit $U$ and some $G_{\nu,r}\in {\mathbb{Q}}[\vec{a}]$ which becomes
independent of $p$ when $p$ is large enough. Note that $G_{\nu,r}$ is
nonconstant since it contains a unique monomial corresponding to the
permutation $\rho_t$ in $S_k$ by composing $\rho_{1,t}\in S_{k_1}$ and $%
\rho_{2,t}\in S_{k_2}$ in the obvious way. Let $ {\mathcal{U}} _{\nu,r}$ be
the subspace of ${\mathbb{A}}_{d_1,d_2}$ defined by $G_{\nu,r}\neq 0$. Hence 
$ {\mathcal{U}} _{\nu,r}$ over ${\mathbb{Q}}$ is open dense in ${\mathbb{A}}%
_{d_1,d_2}$.

\begin{lemma}
\label{L:a-bounds} Write $\det(1-\alpha T) = \sum_{j=0}^{\infty}C_j T^j$.
Fix $1\leq k\leq d_1+d_2$. For $p$ large enough, we have 
\begin{eqnarray}  \label{E:C_k}
C_k & \equiv & \prod_{t=0}^{a-1} \det M_t^{[k]} \bmod \gamma^{>a(p-1)s_k};
\end{eqnarray}
furthermore (with $p$ still large enough), we have 
\begin{eqnarray}  \label{E:ord_q(C_k)}
\mathrm{ord} _q (C_k) &=& \frac{1}{a}\sum_{t=0}^{a-1} \mathrm{ord} _p\det
M_t^{[k]} = s_k
\end{eqnarray}
if and only if $\overline{P}(x) \in  {\mathcal{U}} _{\nu,r}({\mathbb{F}}_q)$.
\end{lemma}

\begin{proof}
It is clear that the $k\times k$ submatrix of $M_t$ whose determinant has
the smallest $p$-adic valuation shares the rows of $M_t^{[k]}$. Let $N$ be
any $ak\times ak$ principal submatrix in $\vec{M}_{[a]}$. Let $N_t$ be the
intersection of $N$ and $M_t$ as submatrix of $\vec{M}_{[a]}$ for all $0\leq
t\leq a-1$. We may well assume that $N_t$ is $k\times k$ matrix as in the
proof of Lemma \ref{L:main}. Now suppose for some $t$ we have $N_t \neq
M_t^{[k]}$ share the rows of $M_t^{[k]}$. Observe that row indices of $N_t$
are equal to the column indices or $N_{t+1}$ because $N$ is principal. Note
that in fact we consider the subindices modulo $a$. Since $N_t$ has at least
one column outside of the columns of $M_t^{[k]}$, we have that $N_{t-1}$ has
at least one row outside of the rows of $M_{t-1}^{[k]}$. Recall that the
difference between minimal row valuations in $M_t$ is $\geq 1/d_1$ (resp. $%
\geq 1/d_2$) as $p$ is large enough, depending on the location of the row in
the matrix blocks. In comparison, the difference between minimal column
valuations in $M_t$ is convergent to $0$ as $p$ approaches $\infty$. As $%
p\rightarrow \infty$, we have by the same argument as that in \cite[Sections
4,5]{Li-Zhu:1}. $ \mathrm{ord} _p (\det N) > a s_k. $ As $C_k$ is the
infinite sum of $\pm \det N$ as $N$ ranges over all such $ak\times ak$
principal submatrices in $\vec{M}_{[a]}$, the above inequality yields our
first congruence relation in (\ref{E:C_k}).

Note that $ \mathrm{ord} _p \det M_t^{[k]} \geq a s_k$ with equality holds
if and only if $\overline{P}\in {\mathcal{U}} _{\nu,r}({\mathbb{F}}_q)$ by
the paragraph above this lemma. Combined with the congruence relation in (%
\ref{E:C_k}), our second assertion in (\ref{E:ord_q(C_k)}) follows.
\end{proof}

Let $ \mathrm{GNP} ({\mathbb{A}}_{d_1,d_2},\chi_s^r;\overline{\mathbb{F}}_p)$
be the generic Newton polygon of twisted exponential sums over $\overline{%
\mathbb{F}}_p$, namely, 
\begin{eqnarray}
\mathrm{GNP} ({\mathbb{A}}_{d_1,d_2},\chi_s^r;\overline{\mathbb{F}}_p) &=&
\inf_{\overline{P}} \; \mathrm{NP} _q(\overline{P}(x),\chi_s^r;{\mathbb{F}}%
_q)
\end{eqnarray}
where $\overline{P}$ ranges over all Laurent polynomials in ${\mathbb{A}}%
_{d_1,d_2}({\mathbb{F}}_q)$ for all ${\mathbb{F}}_q$ in $\overline{\mathbb{F}%
}_p$. This minimum exists by Grothendieck specialization theorem (see \cite%
{Katz} or \cite{Wan:2}).

To simplify notations, we abbreviate $\lim_{p\rightarrow\infty,p\equiv\nu%
\bmod s}(\cdot)$ by $\lim_\nu(\cdot)$.

\begin{proposition}
\label{P:1} Let notations be as above. Fix $s\geq 1$ and $1\leq \nu\leq s-1$
coprime to $s$. Let $0\leq r\leq s-1$.\newline
(a) For $p\equiv \nu\bmod s$ large enough (depending only on $%
d_1,d_2,\nu,\chi_s^r$) we have

\begin{enumerate}
\item[(i)] $ \mathrm{GNP} ({\mathbb{A}}_{d_1,d_2},\chi_s^r;\overline{\mathbb{%
F}}_p)$ exists and it is given by the vertex points $(k,s_k)_{0\leq k\leq
d_1+d_2}$.

\item[(ii)] we have 
\begin{equation*}
\mathrm{NP} (P(x)\bmod  {\mathcal{P}} ,\chi_s^r) \succ  \mathrm{GNP} ({%
\mathbb{A}}_{d_1,d_2},\chi_s^r;\overline{\mathbb{F}}_p) 
\end{equation*}
for any prime $ {\mathcal{P}} $ over $p$ in $\overline{\mathbb{Q}}$; these
two polygons coincide if and only if $P\in  {\mathcal{U}} _{\nu,r}(\overline{%
\mathbb{Q}})$.
\end{enumerate}

(b) For every $P(x)\in {\mathcal{U}} _{\nu,r}(\overline{{\mathbb{Q}}})$ we
have that 
\begin{equation*}
\lim_\nu  \mathrm{NP} (P\bmod  {\mathcal{P}} ,\chi_s^r) =  \mathrm{HS} ({%
\mathbb{A}}_{d_1,d_2},\nu,\chi_s^r) 
\end{equation*}
for any prime $ {\mathcal{P}} $ over $p$ in $\overline{\mathbb{Q}}$.
\end{proposition}

\begin{proof}
(a) Notice that $\varepsilon_{k,p}\rightarrow 0+$ as $p\rightarrow\infty$,
so for $p\rightarrow \infty$, we have 
\begin{eqnarray*}
s_k &\longrightarrow & \frac{k_1(k_1-1)}{2d_1}+\frac{k_1(1-\lambda)}{d_1} + 
\frac{k_2(k_2-1)}{2d_2}+\frac{k_2\lambda}{d_1}
\end{eqnarray*}
from the right. Thus $ \mathrm{GNP} ({\mathbb{A}}_{d_1,d_2},\chi_s^r;%
\overline{\mathbb{F}}_p)$ is indeed given by vertices with coordinates $%
(k,s_k)$ for $0\leq k\leq d_1+d_2$. This proves (i).

Consider the previous lemma \ref{L:a-bounds} and suppose $p$ large enough as
given there. We have $ \mathrm{ord} _q C_k(\vec{a}) \geq s_k$ and the
equality holds if and only if $P\in {\mathcal{U}} _{\nu,r}(\overline{\mathbb{%
Q}})$. On the other hand, for $p$ large enough, $ \mathrm{NP} _q(\overline
P,\chi_s^r;{\mathbb{F}}_q)$ coincides with the $q$-adic Newton polygon of $%
\sum_{j=0}^{d_1+d_2-1}C_jT^j = \det(1-\alpha T) \bmod T^{d_1+d_2}$. Thus $ 
\mathrm{NP} (P\bmod  {\mathcal{P}} ,\chi_s^r) =  \mathrm{NP} _q(\overline{P}%
,\chi_s^r;{\mathbb{F}}_q) \succ  \mathrm{GNP} ({\mathbb{A}}%
_{d_1,d_2},\chi_s^r;\overline{\mathbb{F}}_p)$ and they coincide if and only
if $P\in  {\mathcal{U}} _{\nu,r}(\overline{\mathbb{Q}})$. This proves (ii).

(b) From part (a) we know that for $P \in {\mathcal{U}} _{\nu,r}(\overline{%
\mathbb{Q}})$ the Newton polygon coincides with the generic Newton polygon,
but the latter converges to the Hodge-Stickelberger polygon as $p$
approaches infinity by looking at the limit of $s_k$. This proves (b).
\end{proof}

\section{Newton polygons for Laurent polynomials $\overline{P}(x^s)$.}

\label{S:5}

We shall prove the main theorems in this section. Let $\chi_s$ be a
multiplicative character of order $s$ on ${\mathbb{F}}_{q^k}^\times$. Set $%
n:=\gcd(s,q^k-1)$. Let $\chi_n:=\chi_s^{s/n}$ be a multiplicative character
of ${\mathbb{F}}_{q^k}^\times$ of order $n$.

\begin{lemma}
\label{L:above} Then we have 
\begin{equation*}
S_k(\overline{P}(x^s)) =\sum_{r=0}^{n-1}S_k(\overline{P},\chi_n^r) =
\sum_{r^{\prime}\in\{0,\ldots,s-1\},\frac{s}{n}|r^{\prime}} S_k(\overline{P}%
,\chi_s^{r^{\prime}}) . 
\end{equation*}
\end{lemma}

\begin{proof}
By hypothesis, we may factor $s$ as a product of two integers $s=mn$. Since $%
\gcd(m,q^k-1)=1$, the map $x\mapsto x^{m}$ is bijective on ${\mathbb{F}}%
_{q^k}^\times$. On the other hand, since $n|q^k-1$, the kernel of the map $%
x\mapsto x^n$ is the set of $n$-th roots of unity, and its image is the set $%
({\mathbb{F}}_{q^k}^\times)^n$ of $n$-th powers in ${\mathbb{F}}_{q^k}^\times
$. Thus we get 
\begin{equation*}
S_k(\overline{P}(x^s))=\sum_{x\in {\mathbb{F}}_{q^k}^\times}\psi_{q^k}(%
\overline{P}(x^s)) =\sum_{x\in ({\mathbb{F}}_{q^k}^\times)^n} n\psi_{q^k}(%
\overline{P}(x)). 
\end{equation*}
{}From the orthogonality relations on multiplicative characters, we have
that $\sum_{r=0}^{n-1}\chi_n(x^r) = n$ if $x\in ({\mathbb{F}}_{q^k}^\times)^n
$ and $=0$ otherwise. Then the above equation becomes 
\begin{equation*}
\begin{array}{rcl}
S_k(\overline{P}(x^s)) & = & \sum_{x\in {\mathbb{F}}_{q^k}^\times}
\sum_{r=0}^{n-1}\chi_n(x^r) \psi_{q^k}(\overline{P}(x)) \\ 
& = & \sum_{r=0}^{n-1}\sum_{x\in {\mathbb{F}}_{q^k}^\times} \chi_n(x^r)
\psi_{q^k}(\overline{P}(x)) \\ 
& = & \sum_{r=0}^{n-1}S_k(\overline{P},\chi_n^r).%
\end{array}%
\end{equation*}
The last equation is straightforward.
\end{proof}

Observe that if $s|(q^k-1)$, then we have $S_k(\overline{P}(x^s)) =
\sum_{r=0}^{s-1}S_k(\overline{P},\chi_s^r) . $

Consider the permutation $\sigma^a$ on $\{0,\dots,s-1\}$, namely the
permutation induced by multiplication of $q=p^a$ modulo $s$. Its cycle
decomposition (including $1$-cycles) is further splitting of that of $\sigma$
as $\sigma^a = \prod_{i=1}^{u}\sigma_i^a= \prod_{i=1}^{u}
\prod_{j=1}^{\ell_i/\ell_i^{\prime}}\sigma_{ij}$ for $\ell^{\prime}_i$%
-cycles $\sigma_{ij}$, where $\ell_i^{\prime}|\ell_i$. Namely each
permutation $\sigma_i^a$ splits into $\ell_i/\ell_i^{\prime}$ many cycles of
equal length $\ell_i^{\prime}$.

Consider $\overline{P}\in{\mathbb{A}}_{d_1,d_2}({\mathbb{F}}_q)$ as $%
\overline{P}/{\mathbb{F}}_{q^{\ell^{\prime}_i}}$ for $i=1,\ldots,u$. It is
clear that $s|(q^{\ell_i^{\prime}}-1)$.

For any cycle $\sigma_i$ in the decomposition of $\sigma$ (including $1$%
-cycles), define 
\begin{eqnarray}  \label{E:L-twist2}
L_i(T) &:=& \prod_{j=1}^{\ell_i/\ell_i^{\prime}} L(\overline{P}/{\mathbb{F}}%
_{q^{\ell^{\prime}_i}},\chi_s^{r_{ij}};T^{\ell^{\prime}_i})
\end{eqnarray}
where $r_{ij}$ is an element in $\sigma_{ij}$. This is a polynomial in ${%
\mathbb{Z}}[\zeta_p,\zeta_s][T]$ of degree $\ell_i(d_1+d_2)$.

\begin{lemma}
\label{L:L-functionsplit} We have 
\begin{eqnarray}  \label{E:L1}
L(\overline{P}(x^s);T) &=&\prod_{i=1}^{u} L_i(T).
\end{eqnarray}
\end{lemma}

\begin{proof}
Since $x\mapsto x^q$ is an automorphism of ${\mathbb{F}}_{q^n}^\times$ for
any $n$ and $\overline{P}(x^q)=\overline{P}(x)^q$, we have 
\begin{eqnarray*}
S_k(\overline{P}(x),\chi_n^{r}) &=& S_k(\overline{P}(x^q),\chi_n^r) \\
&=& \sum_{x\in {\mathbb{F}}_{q^k}^\times}\chi_n^r (x^{q}) \psi_{q^k}(%
\overline{P}(x^q)) \\
&=& \sum_{x\in{\mathbb{F}}_{q^k}^\times}\chi_n^{r}(x^q) \psi_{q^k}(\overline{%
P}(x)).
\end{eqnarray*}
This shows that $S_k(\overline{P},\chi_n^r)= S_k(\overline{P},\chi_n^{qr}).$
Consequently the sum in Lemma \ref{L:above} may be broken down into orbits
of $\sigma^a$. Recall $\sigma^a=\prod_{i=1}^{u}\prod_{j=1}^{\ell_i/\ell_i^{%
\prime}}\sigma_{ij}$. Let $r_{ij}$ be an element in $\sigma_{ij}$. Since $%
\ell_i^{\prime}|k$ is the same as saying $\sigma^{ak}(r_{ij})= r_{ij}$, that
is $q^kr_{ij}\equiv r_{ij}\bmod s$. But $s | (q^k-1)r_{ij}$ (combined with
our hypothesis $n=\gcd(s,q^k-1)$) is equivalent to $\frac{s}{n}|r_{ij}$.
Thus the sum in Lemma \ref{L:above} can be phrased as 
\begin{equation*}
S_k(\overline{P}(x^s)) =\sum_{r_{ij},\frac{s}{n}|r_{ij}} \ell_i^{\prime}\;
S_k(\overline{P},\chi_s^{r_{ij}}) =\sum_{r_{ij},\ell_i^{\prime}|k}
\ell_i^{\prime}\; S_k(\overline{P},\chi_s^{r_{ij}}) 
\end{equation*}
where $r_{ij}$ runs in all distinct cycles $\sigma_{ij}$ in $\sigma^a$.
Substitute this identity to twisted $L$-function defined in Section \ref%
{S:3.1}, we get after some elementary computation 
\begin{equation*}
L(\overline{P}(x^s);T) =\prod_{r_{ij}} L(\overline{P}/{\mathbb{F}}%
_{q^{\ell^{\prime}_i}},\chi_s^{r_{ij}};T^{\ell^{\prime}_i}) 
\end{equation*}
where the product ranges over all distinct cycles in $\sigma^a$. Group this
product in terms of cycle decomposition of $\sigma$, we finish our proof.
\end{proof}

\begin{proof}[Proof of Theorem \protect\ref{T:1}]
By Proposition \ref{P:twisted} (ii) we have 
\begin{eqnarray*}
\mathrm{NP} _q(L(\overline{P}/{\mathbb{F}}_{q^{\ell_i^{\prime}}},%
\chi_s^{r_{ij}};T^{\ell_i^{\prime}})) &=&\ell_i^{\prime}\; \mathrm{NP}
_{q^{\ell_i^{\prime}}}(L(\overline{P}/{\mathbb{F}}_{q^{\ell_i^{\prime}}},%
\chi_s^{r_{ij}};T) \\
&\succ & \ell_i^{\prime}\; \mathrm{HS} ({\mathbb{A}}_{d_1,d_2},\nu,%
\chi_s^{r_{ij}}) \\
&=&\ell_i^{\prime}\; \mathrm{HS} ({\mathbb{A}}_{d_1,d_2},\nu,\chi_s^{r_i})
\end{eqnarray*}
for all $1\leq j\leq \ell_i/\ell_i^{\prime}$. Thus 
\begin{equation*}
\mathrm{NP} _q(L_i(T)) \succ \frac{\ell_i}{\ell_i^{\prime}}
(\ell_i^{\prime}\; \mathrm{HS} ({\mathbb{A}}_{d_1,d_2},\nu,\chi_s^{r_i}))
=\ell_i\; \mathrm{HS} ({\mathbb{A}}_{d_1,d_2},\nu,\chi_s^{r_i}). 
\end{equation*}
By the split of the $L$-function in Lemma \ref{L:L-functionsplit} and by
Lemma \ref{L:HS-split}, we have 
\begin{eqnarray*}
\mathrm{NP} _q(\overline{P}(x^s);{\mathbb{F}}_q) &=& \boxplus_{i=1}^{u} 
\mathrm{NP} _q(L_i(T))  \notag \\
&\succ & \boxplus_{i=1}^{u} \ell_i\; \mathrm{HS} ({\mathbb{A}}%
_{d_1,d_2},\nu,\chi_s^{r_i})= \mathrm{HS} ({\mathbb{A}}_{d_1,d_2},\nu,s).
\end{eqnarray*}
where the box-sum ranges over all cycle decomposition of $%
\sigma=\prod_{i=1}^{u}\sigma_i$. The proof for the case $\nu=1$ is omitted
here since it is done in \cite{Zhu:3}.
\end{proof}

Let $ \mathrm{GNP} ({\mathbb{A}}_{d_1,d_2},s;\overline{\mathbb{F}}_p)$ be
the generic Newton polygon for exponential sums of $\overline{P}%
(x^s)/\overline{\mathbb{F}}_p$. That is, 
\begin{equation*}
\mathrm{GNP} ({\mathbb{A}}_{d_1,d_2},s;\overline{\mathbb{F}}_p):=\inf_{%
\overline{P}} \mathrm{NP} _q(\overline{P}(x^s);{\mathbb{F}}_q) =\inf_{P} 
\mathrm{NP} (P(x^s)\bmod {\mathcal{P}} ) 
\end{equation*}
where $\overline{P}$ ranges in ${\mathbb{A}}_{d_1,d_2}({\mathbb{F}}_q)$ for
any $q$, and where $P$ ranges in ${\mathbb{A}}_{d_1,d_2}(\overline{\mathbb{Z}%
}_p\cap\overline{\mathbb{Q}})$ and $ {\mathcal{P}} $ is any prime over $p$
in $\overline{\mathbb{Q}}$.

Recall $\sigma$ is a permutation on $\{0,1,\ldots,s-1\}$ induced by
multiplication by $p$ modulo $s$. For every cycle $\sigma_i$ in $\sigma$ we
have a nonconstant polynomial $G_{\nu,r}$ (see Proposition \ref{P:1}) where $%
r$ is an element in $\sigma_i$. (It is independent of the choice of $r$ in $%
\sigma_i$.) Let $G_\nu=\prod G_{\nu,r}$ where $r$ runs in distinct cycles $%
\sigma_1,\ldots, \sigma_u$ of $\sigma$, then $G_\nu$ is nonconstant
polynomial in ${\mathbb{Q}}[\vec{a}]$ as well. Let $ {\mathcal{U}} _\nu$ be
the complement of $G_\nu=0$ in ${\mathbb{A}}_{d_1,d_2}$. Then $ {\mathcal{U}}
_\nu$ is a Zariski dense open subset of ${\mathbb{A}}_{d_1,d_2}$ defined
over ${\mathbb{Q}}$. Our Theorems \ref{T:1} and \ref{T:2} are proved in the
following stronger version. Its proof is similar to that of Theorem \ref{T:1}%
.

\begin{theorem}
\label{T:4} Let notations be as in Theorem \ref{T:2}. Then\newline
(a) For $p\equiv \nu\bmod s$ large enough (depending only on $d_1,d_2,\nu,s$%
), we have

\begin{enumerate}
\item[(i)] $ \mathrm{GNP} ({\mathbb{A}}_{d_1,d_2},s;\overline{\mathbb{F}}_p)$
exists and $ \mathrm{NP} (P(x^s)\bmod  {\mathcal{P}} ) \succ  \mathrm{GNP} ({%
\mathbb{A}}_{d_1,d_2},s;\overline{\mathbb{F}}_p)$ for all $P\in{\mathbb{A}}%
_{d_1,d_2}(\overline{\mathbb{Q}})$;

\item[(ii)] these two polygons coincide if and only if $P\in  {\mathcal{U}}
_\nu(\overline{\mathbb{Q}})$.
\end{enumerate}

(b) For $P\in  {\mathcal{U}} _\nu(\overline{\mathbb{Q}})$ we have 
\begin{equation*}
\lim_\nu  \mathrm{NP} _q(P(x^s)\bmod  {\mathcal{P}} ;{\mathbb{F}}_q) =  
\mathrm{HS} ({\mathbb{A}}_{d_1,d_2},\nu,s). 
\end{equation*}
\end{theorem}

\begin{proof}
Our theorem follows immediately by applying Proposition \ref{P:1} and the
key Lemma \ref{L:L-functionsplit} in the same fashion as that in the proof
of Theorem \ref{T:1}. We hence omit details here.
\end{proof}

Finally we remark in the polynomial case, i.e., $d_2=0$, similar argument
can be carried out which yield similar results. In fact, one can carry out
calculations in the spirit of \cite{B-F} to get explicitly the generic Newton 
polygons, and \emph{Hasse polynomials} that describe exactly which polynomials
attain this polygon.

\section{Further questions and multivariable cases}

\label{S:6}


\subsection{Global permutation polynomials}

Wan's \cite[Conjecture 1.12]{Wan:2} was proved in 1-variable case by \cite%
{Zhu:1,Zhu:2} and was generalized to Laurent polynomials in \cite{Li-Zhu:1}
that there is a Zariski dense open subset $ {\mathcal{U}} $ in ${\mathbb{A}}%
_{d_1,d_2}$ defined over ${\mathbb{Q}}$ such that for every $f\in  {\mathcal{%
U}} (\overline{{\mathbb{Q}}})$ we have its limit of Newton polygon
approaching the Hodge polygon as $p\rightarrow \infty$. It has been
fascinating researchers to know what (Laurent) polynomials $f$ over $%
\overline{\mathbb{Q}}$ that would fail the asymptotic property $%
\lim_{p\rightarrow\infty} \mathrm{NP} (f\bmod  {\mathcal{P}} ) = \mathrm{HP}
({\mathbb{A}}_{d_1,d_2})$. We will discuss below some known such (Laurent)
polynomials. For simplicity we restrict ourselves over ${\mathbb{Q}}$
instead of extension of ${\mathbb{Q}}$ in $\overline{\mathbb{Q}}$, one can
extend our argument to extensions of ${\mathbb{Q}}$ by the references we
shall provide in the context.

For any positive integer $n$, let $D_n(x,y)$ be the unique polynomial in ${%
\mathbb{Z}}[x,y]$ such that $D_n(u+v,uv)=u^n+v^n$. For any $c\in {\mathbb{Q}}
$ the monic degree-$n$ polynomial $D_n(x,c)$ in ${\mathbb{Q}}[x]$ is called
a degree-$n$ \emph{Dickson polynomial} over ${\mathbb{Q}}$. If $p$ divides $c
$, then $D_n(x,c)=x^n$ is a monomial which is a permutation on ${\mathbb{F}}%
_p$ if and only if $\gcd(n,p-1)=1$; If $p$ does not divide $c$, it is a
permutation on ${\mathbb{F}}_p$ if and only if $\gcd(n,p^2-1)=1$ (due to 
\cite{Dickson}, see \cite[Chapter 7]{Lidl} for quick reference).

For any $l\geq 1$, let \emph{global permutation polynomial over ${\mathbb{Q}}
$ of level $l$} be a polynomial $h(x)$ in ${\mathbb{Q}}[x]$ such that $%
x\mapsto h(x)$ is a permutation on ${\mathbb{F}}_p, \ldots, {\mathbb{F}}%
_{p^l}$ for infinitely many primes $p$. It is easy to see that $D_n(x,c)$ in 
${\mathbb{Q}}[x]$ is a global permutation polynomial of level $l$ if and
only if every prime factor $Q$ of $n$ satisfies $Q>l+1$ (when $c=0$) and $Q
> 2l+1$ (when $c\neq 0$). Thus for level $l=1$ it is equivalent to $2\nmid n$
(when $c=0$) and $\gcd(n,6)=1$ (when $c\neq 0$).

It is proved by Schur that every global permutation polynomial over ${%
\mathbb{Q}}$ is a composition of Dickson polynomials $D_n(x,c)$ over ${%
\mathbb{Q}}$ and linear polynomials over certain extensions of ${\mathbb{Q}}$%
. (This is generalized to all number fields by Fried in \cite{Fried}.)

Our result in Theorem \ref{T:2} implies that for any polynomial or Laurent
polynomial $f(x)$ over $\overline{\mathbb{Q}}$ containing $x^s=D_s(x,0)$ as
a right composition factor for any $s>2$, that is, $f(x) = P(x^s)$, the
limit of $p$-adic Newton polygon does not exist as $p\rightarrow \infty$.
Following Wan's argument on polynomials which is communicated to the
authors, we demonstrate here that if $f(x)$ is any Laurent polynomial in ${%
\mathbb{A}}_{d_1,d_2}(\overline{\mathbb{Q}})$ containing a global
permutation polynomial of degree $s>1$ of level $3$ as a right composition
factor, that is $f(x)=P(D_s(x,c))$, then the limit $\lim_{p\rightarrow
\infty}  \mathrm{NP} (f(x)\bmod p)$ does not exist. Without loss of
generality, we assume $d_1\geq d_2$ for the rest of this paragraph. Since $s$
must be odd, our Dickson polynomials fixes $0$ and $\infty$, and finally we
assume the global permutation polynomial composition factor is $D_s(x,c)$
for some $s>1$ where $s$'s prime factors are all $\geq 7$ (when $c=0$) and $%
\geq 11$ (when $c\neq 0$). Write $L(f(x);{\mathbb{F}}_p) =
1+C_1T+C_2T^2+\cdots$ and $L(P(x);{\mathbb{F}}_p)=1+c_1T+\cdots$. For any
prime $p$ such that $D_s(x,c)$ permutes ${\mathbb{F}}_p,{\mathbb{F}}_{p^2},{%
\mathbb{F}}_{p^3}$, we have $S_k(f;{\mathbb{F}}_p) = S_k(P;{\mathbb{F}}_p)$
for $1\leq k\leq 3$. By the lower bound for Newton polygon of $L(P(x);{%
\mathbb{F}}_p)$ (see \cite{Zhu:2}) we have $ \mathrm{ord} _p C_2= \mathrm{ord%
} _p c_2\geq 1/(d_1/s) = s/d_1 \geq 7/d_1$. This implies the Newton polygon
of $L(f(x);{\mathbb{F}}_p)$ does not have a breakpoint at $(2, 1/d_1)$;
similarly, since $ \mathrm{ord} _p c_3\geq 2s/d_1\geq 14/d_1$ a breakpoint
at $(3,3/d_1)$ is impossible. On the other hand, we know that for infinitely
many prime $p$ (precisely those $p\equiv 1\bmod  \mathrm{lcm } (sd_1,sd_2)$) 
$ \mathrm{NP} (f\bmod p)$ coincides with its lower bound and
has break point at $(2,1/d_1)$ if $d_1>d_2$ and at $(3,3/d_1)$ if $d_1=
d_2$. Thus $\lim_{p\rightarrow\infty} \mathrm{NP} (f\bmod p)$ does not exist.

We say two (Laurent) polynomials $f(x)$ and $h(x)$ over ${\mathbb{Q}}$ of
degree $d$ are \textit{Artin-Schreier isomorphic} if the two Artin-Schreier
curves given by $f(x)=h(wx+v)$ for some $d$-th root of unity $w$ and $%
v\in\overline{\mathbb{Q}}$. For reader's convenience, we quote a corrected
version of Wan's conjecture below from \cite[Chapter 5]{Yang}.

\begin{conjecture}[Wan]
If $f(x)$ is a polynomial in ${\mathbb{Q}}[x]$ which does not contain a
global permutation polynomial of degree $>1$ as right composition factor
over ${\mathbb{Q}}$ (upto Artin-Schreier isomorphism), then $%
\lim_{p\rightarrow \infty}  \mathrm{NP} _p(f\bmod p)$ exists and is equal to
its lower bound Hodge polygon.
\end{conjecture}

\subsection{A variant of Schur's theorem}

Let $\psi: {\mathbb{F}}_p\rightarrow {\mathbb{Q}}(\zeta_p)^\times$ be the
nontrivial additive character defined by $\psi(a)= \zeta_p^a$. 

\begin{conjecture}
\label{Conj:2} Let $f(x)\in{\mathbb{Q}}[x]$ be of degree $d\geq 2$ and let 
$S(f(x)\bmod p)=\sum_{x\in {\mathbb{F}}_p}\psi(f(x))$ be the first
exponential sum mod $p$. Let $\varepsilon>0$. If $ \mathrm{ord} _p S(f(x)
\bmod p) > 1/d +\varepsilon$ for infinitely many primes $p$, then $f(x)=
P(D_s(x,c))$ (up to Artin-Schreier isomorphism) for some 
$P\in {\mathbb{Q}}[x]$ and a global permutation Dickson polynomial $D_s$ of degree $s>1$.
\end{conjecture}

The conjecture above can be considered as a generalization of the Schur's
conjecture on global permutation polynomials since it can be phrased in the
following term: ``For any $f\in {\mathbb{Q}}[x]$ if $S(f(x)\bmod p) = 0$
(i.e.,$ \mathrm{ord} _p(S(f(x)\bmod p))=+\infty$) for infinitely many
prime $p$ then $f(x)$ is a Dickson polynomial up to Artin-Schreier
isomorphism'' (see \cite[Chapter 7]{Lidl}).

\begin{proposition}
Let notation be as above and suppose Conjecture \ref{Conj:2} holds. Then the
limit $\lim_{p\rightarrow\infty} \mathrm{NP} _1(f(x)\bmod p)$ of first slope
exists if and only if $f(x)\neq P(D_s(x,c))$ (up to any Artin-Schreier
isomorphism) for some $P\in{\mathbb{Q}}[x]$ and a global permutation Dickson
polynomial $D_s(x)$ of degree $s>1$.
\end{proposition}

\begin{proof}
It was already proved above that if $f(x)$ contains a right Dickson
composition factor (of degree prime to $2$ or $6$ depending on whether $c=0$ or
not) then the limit does not exist. Conversely, suppose the first slope
limit does not exist. Since for $p\equiv 1\bmod d$ we always have $ \mathrm{%
NP} _1(f\bmod p) = 1/d$ (and $ \mathrm{NP} _2(f\bmod p)=2/d$), this is
equivalent to the hypothesis of Conjecture \ref{Conj:2} since $ \mathrm{NP}
_1(f\bmod p) =  \mathrm{ord} _p S_1(f\bmod p)$.
\end{proof}

\subsection{Multivariable cases}

Our main result in Theorem \ref{T:1} generalizes to multivariable cases. Let 
${\mathbb{A}}$ be the space of polynomials in $n$ variables $x_1,\ldots,x_n$
parametrized by their coefficients of monomials. Let $\overline{P}$ be a
polynomial in ${\mathbb{A}}({\mathbb{F}}_q)$. Fix $\vec{s}=(s_1,\ldots,s_n)$
for integers $s_\iota\geq 1$. All primes $p$ in this subsection will be
coprime to $s_1\cdots s_n$. Let $\vec\nu = p\bmod \vec{s}$, the least
nonnegative residue. For each $1\leq \iota\leq n$, let $\sigma_\iota$ be the
permutation on the set $\{0,1,\ldots,s_%
\iota-1\}$ induced by multiplication of $p$. We write its cycle
decomposition as 
\begin{equation*}
\sigma_\iota = \prod_{i_\iota=1}^{u_\iota}\sigma_{\iota,i}
\end{equation*}
for $\ell_{\iota,i}$-cycles $\sigma_{\iota,i}$ (including $1$-cycles!). For
each $1\leq \iota\leq n$ and $1\leq i_\iota\leq u_\iota$, let 
\begin{equation*}
\lambda_{\iota,i_\iota} := \frac{\sum_{j\in \sigma_{\iota,i_\iota}}j} {%
s_\iota \ell_{\iota,i_{\iota}}}. 
\end{equation*}
So $0\leq \lambda_{\iota,i_{\iota}} <1$. Write $\vec\lambda_{\vec{i}%
}:=(\lambda_{1,i_1},\ldots,\lambda_{n,i_n})$. Let $w:{\mathbb{Z}}%
^n\rightarrow {\mathbb{Z}}$ be the weight function with respect to a given $%
\overline{P}$ as in \cite{AS:1} and \cite{Wan:2}. It is easy to see that it
extends to ${\mathbb{Q}}^n$.

We define the \emph{Hodge-Stickelberger polygon} $ \mathrm{HS} ({\mathbb{A}},%
\vec{\nu},\vec{s})$ in multivariable setting as concatenation of line
segments given by 
\begin{equation*}
(w(\vec{m}-w(\vec\lambda_{\vec{i}})), \ell_{1,i_1}\cdots\ell_{n,i_n}) 
\end{equation*}
where $\vec{m}$ ranges over the $n!V(\overline{P})$ elements in $C(\overline{%
P})\cap {\mathbb{Z}}^n$ as defined in \cite{AS:1}, $1\leq \iota\leq n$ and $%
1\leq i_\iota\leq u_\iota$. One observes that this polygon has horizontal
length $s_1\cdots s_n n! V(\overline{P})$.

\begin{theorem}
\label{T:5} Suppose $\overline{P}(x_1,\ldots,x_n)$ over ${\mathbb{F}}_q$ is
nondegenerate and the dimension of the polyhedrum $\Delta(\overline{P})$ is
equal to $n$, then $L(\overline{P}(x_1^{s_1},\ldots,x_n^{s_n})/{\mathbb{F}}%
_q,T)^{(-1)^{n-1}}$ is a polynomial. Moreover its Newton polygon lies over $ 
\mathrm{HS} ({\mathbb{A}},\vec{\nu},\vec{s})$ and their endpoints meet.
\end{theorem}

The proof of this theorem is parallel to the proof of Theorem \ref{T:1} and
will introduce lots more notations and we hence omit it here. We want to
emphasize here that Theorem \ref{T:5} does not include Theorem \ref{T:1} as
a corollary. It is slightly weaker in the one-variable special case.

Finally we remark that the asymptotic result in Theorem \ref{T:2} seems
harder to generalize. Nevertheless, from Theorem \ref{T:5} one observes
already that for each residue class $\vec{\nu}= p \bmod \vec{s}$ there is a
distinct lower bound $ \mathrm{HS} ({\mathbb{A}},\vec\nu,\vec{s})$. So one
can not expect there is a limit on the generic Newton polygon as $%
p\rightarrow \infty$.

\begin{acknowledgments}
The authors thank Daqing Wan for invaluable communication regarding his
conjecture(s) and for pointing out the reference \cite{Yang} to us. We also
thank Michael Zieve for providing us with an interesting account of references
and history on Dickson polynomials.
\end{acknowledgments}

\end{document}